\documentclass[A4,12pt]{svjour3}
\usepackage{amssymb,amsmath,latexsym,color,
amsfonts}
\usepackage{hyperref}
\usepackage{mathrsfs}
\usepackage[dvipsnames]{xcolor}

\newcommand{\real}{\mathbb R}

\numberwithin{equation}{section} 

\DeclareMathOperator*{\esssup}{ess\,sup}
\DeclareMathOperator*{\essinf}{ess\,inf}

\usepackage{hyperref}

\newcommand{\la}{\langle}
\newcommand{\ra}{\rangle}
\newcommand{\ol}{\overline}
\newcommand{\ul}{\underline}

\begin{document}
\title{Multivalued nonmonotone dynamic boundary condition}
\titlerunning{}        

\author{Khadija Aayadi \and Khalid Akhlil \and Sultana Ben Aadi \and Mourad El Ouali}


\institute{ 
Khadija Aayadi, FPO, Ibn Zohr University, khadija.aayadi@gmail.com\\
Khalid Akhlil, FPO, Ibn Zohr University, akhlil.khalid@gmail.com\\
Sultana Ben Aadi, Ibn Zohr University, sultana.benaadi@edu.uiz.ac.ma\\
Mourad El Ouali, FPO, Ibn Zohr University, mauros1608@gmail.com}

\date{Received: date / Accepted: date}
\maketitle

\begin{abstract}
In this paper, we introduce a new class of hemivariational inequalities, called dynamic boundary hemivariational inequalities reflecting the fact that the governing operator is also active on the boundary. In our context, it concerns the Laplace operator with Wentzell (dynamic) boundary conditions perturbed by a multivalued nonmonotone operator expressed in terms of Clarke subdifferentials. We will show that one can reformulate the problem so that standard techniques can be applied. We will use the well-established theory of boundary hemivariational inequalities to prove that under growth and general sign conditions, the dynamic boundary hemivariational inequality admits a weak solution. Moreover, in the situation where the functionals are expressed in terms of locally bounded integrands, a "filling in the gaps" procedure at the discontinuity points is used to characterize the subdifferential on the product space. Finally, we prove that, under a growth condition and eventually smallness conditions,  Faedo-Galerkin approximation sequence converges to a desired solution.
\end{abstract}



\tableofcontents

\section{Introduction}
Let $\Omega\subset\mathbb R^d$ with enough smooth boundary $\Gamma$, and $A=-\nabla.(a\nabla)$ be a uniformly elliptic differential operator. Then, the operator $A$ with Wentzell boundary conditions is given by the system

\begin{equation}\label{wentz}
 \left\{
  \begin{array}{ll}
    Au=f & \hbox{ in } \Omega; \\[0.2cm]
    -Au+b\,\partial_\nu^a\,u+c u=0 & \hbox{ on } \Gamma
  \end{array}
\right.
\end{equation}where $b$ and $c$ are nonnegative bounded measurable functions on $\Gamma $ and $\partial_\nu^a\,u=(a\nabla u).\nu$ is the co-normal derivative of $u$ with respect to $a$. The fact that the boundary condition in \eqref{wentz} could involve the operator $A$ ( eventually with jumps) goes back to the pioneering paper of Wentzell \cite{W59}. What make Wentzell boundary condition relevant in the applications is the time-derivative introduced in the boundary conditions. More explicitly, the heat equation with the Wentzell boundary condition becomes
\begin{equation}\label{evol}
 \left\{
  \begin{array}{ll}
   u'+A u=f_1 & \quad\hbox{ in } [0,+\infty[\times\Omega ; \\[0.1cm]
    u'+b\partial_\nu^a \,u+cu=f_2 & \quad\hbox{ on }
[0,+\infty[\times\Gamma;\\[0.1cm]
u(x,0)=u_0(x).
  \end{array}
\right.
\end{equation}

The heat equation \eqref{evol} corresponds to the situation where there is a heat source acting on the boundary \cite{FGGR02,FGGR06}. Moreover, in the model of vibrating membrane, the Wentzell boundary condition arises if we assume the boundary $\Gamma$ can be affected by vibrations in $\Omega$ and thus contributes to the kinetic energy of the system \cite{C75,L32}. In \cite{MW28} dynamic boundary conditions are derived for a solid in contact with thin layer of stirred liquid with a heat exchange coefficient represented here by $c$. A detailed derivation of dynamic boundary conditions can also be found in \cite{G06} and references therein. In many physical situations the exchange rate of heat diffusion can be nonlinear or nonmonotone. In this situation, the Wentzell boundary problem becomes
\begin{equation}\label{evol2}
 \left\{
  \begin{array}{ll}
   u'+A u+\gamma_1(u)\ni f_1 & \quad\hbox{ in } [0,+\infty[\times\Omega \\[0.1cm]
    u'+b\,\partial_\nu^a \,u+\gamma_2(u)\ni f_2 & \quad\hbox{ on }
[0,+\infty[\times\Gamma;\\[0.1cm]
u(x,0)=u_0(x).
  \end{array}
\right.
\end{equation}
The nonlinear, eventually nonmonotone, dynamic boundary conditions have been extensively studied in recent years. In the case where $\gamma_1=\partial \phi_1$ and $\gamma_2=\partial \phi_2$ are the subdifferentials, in the sense of convex analysis, of proper, convex and lower semicontinuous functionals $\phi_1$ and $\phi_2$ respectively (hence with maximal monotone graphs), the problem \eqref{evol2} generates, by a result from Minty \cite{M62}, a unique solution described by a strongly continuous nonlinear semigroup\cite{FGGR02,FGGR06,W07,W12,W13}. In the case of nonmonotone but single valued $\gamma_1$ and $\gamma_2$ such problems were considered in \cite{GW10,G12,G15}, and references therein. The functions $\gamma_1$ and $\gamma_2$ was supposed to be of class $C^1$ and satisfy a sign-growth conditions in the spirit of critical point theory.

The main tool in studying the above problems is to work on a product space instead of the state space itself. This trick, now a standard procedure, provide a good insight into the structure of the problem. In fact, consider functions of the form $U=(u,u_{|\Gamma})$ defined on a suitable product space and define the operator $\mathcal A$ by $\mathcal A U=(Au,b\,\partial_\nu^a\,u+cu)$, then the problem \eqref{evol} can be formulated as follows 
\begin{equation}\label{evol3}
 \left\{
  \begin{array}{ll}
    U'+\mathcal A U=f \\[0.2cm]
U(0)= U_0.
  \end{array}
\right.
\end{equation}where $f=(f_1,f_1)$. The idea to incorporate boundary conditions into a product space goes back to Greiner \cite{G87} and has been used by Amann and Escher \cite{AE96}, Arendt \textit{et al}. \cite[Chapter 6]{ABHN01}  and in \cite{VV03} in the context of Dirichlet forms. In the context of heat equations, Wentzell boundary conditions were introduced by A. Favini \textit{et al.} \cite{FGGR02}, see also \cite{CFGGGOR09,MR06}.

In the situation of nonlinear multivalued dynamic boundary conditions \eqref{evol2}, the product space procedure leads to 
\begin{equation}\label{evol4}
 \left\{
  \begin{array}{ll}
    U'+\mathcal A U+\partial \phi_1(u)\times\partial \phi_2(u_{|\Gamma})\ni f \\[0.2cm]
U(0)=U_0.
  \end{array}
\right.
\end{equation}
In the framework of convex functionals, the regularity, in the sense of \cite{C75}, of $\phi_1$ or $\phi_2$ at all points allows us to write the inclusion 
\begin{equation}\label{inclusion}
\partial \phi(U)\subset \partial \phi_1(u)\times\partial \phi_2(u_{|\Gamma})
\end{equation}
where the functional $\phi$ is defined by $\phi(U)=\phi_1(u)+\phi_2(u_{|\Gamma})$. Moreover, by Proposition \ref{partial2} the equality holds in \eqref{inclusion}. This is due to the structure of $\phi$ as a variable separated functional. The problem \eqref{evol4} is then equivalent to the following problem
\begin{equation}\label{evol7}
 \left\{
  \begin{array}{ll}
    U'+\mathcal A U+\partial \phi(U)\ni f \\[0.2cm]
U(0)=U_0.
  \end{array}
\right.
\end{equation}
and can be solved by the nonlinear semigroups theory or the variational inequality theory. In the nonconvex functionals framework, the inclusion \eqref{inclusion} still holds but not the equality except for the regular case. Hence, generally, one can only say that the solvability of \eqref{evol7} implies the solvability of \eqref{evol4}, which, in addition, can not be expressed as a variational inequality due to a lake of monotonicity.

It is the aim of this paper to prove that the problem \eqref{evol7} have a weak solution by using the theory of hemivariational inequality. This theory was initiated by Panagiotopoulus as a generalization of the variational inequality theory \cite{P88,P89,P12}. Hemivariational inequalities are suitable to model physical and engineering problems where multivalued and nonmonotone constitutive laws are involved(cf., e. g. \cite{1MBA20,2MBA20} and references therein). The main tool in this formulations is the generalized gradient of Clarke and Rockafellar \cite{C83,C75,Rock80}. As a subclass of this theory, we can mention boundary hemivariational inequalities, which are boundary value problems where the boundary condition is multivalued, nonmonotone and of subdifferential form (cf. \cite{MO04,M01,PH08} and references therein). Other aspects including evolution inclusions or boundary conditions which are multivalued, nonmonotone and of subdifferential form can be found in \cite{BGK18,BGLS18,GK16,GOS15,GOS16}.

In this paper we introduce a new class of problems within this theory. It concerns dynamic boundary conditions where the boundary condition is multivalued, nonmonotone and of the subdifferential form. The dynamic boundary hemivariational inequalities can model problems where the boundary contains a thermostat regulating the temperature within certain specified bounds. It can also be incorporated into Navier-Stokes equations to obtain a variant of Boussinesq model describing the behaviour of a heat conducting liquid with boundaries participating in the total energy. The exchange of the heat with the boundary can then be expressed, in a general way, with a multivalued, nonmonotone and of the subdifferential form functional.

The structure of this paper is as follows. In Section 2, we present the preliminary material needed later. In section 3, we state our problem in a suitable functional spaces and we prove the existence of weak solutions in section 4. In section 5 we present a seemingly new result related to partial generalized gradient for nonregular locally Lipschitz  functions and we prove a Chang's type lemma related to locally bounded functions. Finally, we devote section 6 to the convergence of the Faedo-Galerkin approximation to desired solutions.

\section{Preliminaries}

Let $E$ be a reflexive Banach space with its dual $E^*$ and $A:D(A)\subset E\rightarrow 2^{E^*}$ be a multivalued function, where \[D(A)=\{u\in E:\, Au\neq\emptyset\}\] stands for the domain of $A$. We say that $A$ is \textit{monotone} if $\la u^*-v^*,u-v\ra_{E^*\times E}\geq 0$ for all $u^*\in Au$,\, $v^*\in Av$  and  $u,\,v\in D(A)$. If moreover, $A$ has a maximal graph in the sense of inclusion among all monotone operators, then we say that $A$ is maximal monotone. In the theory of multivalued nonlinear inclusions, we may apply a multivalued perturbation. This can be made with respect to \textit{pseudomonotone} operators, that is, operators satisfying the following properties,
\begin{enumerate}
\item[(a)] for each $u\in E$, the set $Au$ is nonempty, closed and convex in $E^*$.
\item[(b)] $A$ is upper semicontinuous from each finite dimensional subspace of $E$ into $E^*$ endowed with its weak topology;
\item[(c)] if $u_n\rightarrow u$ weakly weakly in $E$, $u_n^*\in Au_n$ and $\limsup_{n\to\infty}\limits\,\la u_n^*,u_n-u\ra_{E^*\times E}\leq 0$, then for each $v\in E$ there exists $v^*\in Au$ such that $\la v^*,u-v\ra_{E^*\times E}\leq \liminf_{n\to\infty}\limits\,\la u_n^*,u_n-v\ra_{E^*\times E} $.
\end{enumerate}
When dealing with evolution inclusions, another concept of pseudomonotonicity should be introduced. In fact, a large class of evolution inclusions can be written as the sum of a maximal monotone operator resulting from the time derivative and a multivalued operator. This pseudomonotonicity should be defined with respect to the maximal monotone operator as follows: an operator  $A$ is \textit{pseudomonotone with respect to $D(L)$}(or $L-$pseudomonotone), if for a linear, maximal monotone operator $L:D(L)\subset E\rightarrow E^*$, if $(a)$ and $(b)$ are satisfied and
\begin{enumerate}
\item[(c')] for each sequences $\{u_n\}\subset D(L)$ and $\{u_n^*\}\subset E^*$ with $u_n\rightarrow u$  weakly in  $E$, $Lu_n\rightarrow Lu$  weakly in  $E^*$,  $u_n^*\in Au_n$  for all $n\in\mathbb N$, $u_n^*\rightarrow u^*$ weakly in $E^*$ and  $\displaystyle\limsup_{n\to+\infty}\la u_n^*,u_n-u\ra_{E^*\times E}\leq 0$, we have $u^*\in Au$ and $\displaystyle\lim_{n\to+\infty}\la u_n^*,u_n\ra_{E^*\times E}=\la u^*,u\ra_{E^*\times E}$.
\end{enumerate}

$A$ is \textit{coercive}  if there exists a function $c:\mathbb R^+\rightarrow\mathbb R$ with $c(r)\rightarrow\infty$ as $r\to\infty$ such that $\la u^*,u\ra_{E^*\times E}\geq c(\|u\|_E)\|u\|_E$ for every $(u,u^*)\in\mathrm{Graph}(A)$.

For a single-valued operator $A:E\rightarrow E^*$, we say that $A$ is \textit{demicontinuous} if it is continuous from $E$ to $E^*$ endowed with weak topology and $A$ \textit{pseudomonotone} if for each sequence $\{u_n\}\subset E$ such that it converges weakly to $u\in E$ and $\displaystyle\limsup_{n\to\infty}\la Au_n,u_n-u\ra_{E^*\times E}\leq 0$, we have $\la Av,u-v\ra_{E^*\times E}\leq \displaystyle\liminf_{n\to\infty}\la A u_n,u_n-v\ra_{E^*\times E}$ for all $v\in E$.

Now let $\varphi:E\rightarrow \ol{\mathbb R}:=\mathbb R\cup\{+\infty\}$ be a proper, convex and lower semicontinuous functional. The mapping $\partial_c\varphi:E\rightarrow 2^{E^*}$ defined by 
\[
\partial_c\varphi(u)=\{u^*\in E^*:\,\la u^*,v-u\ra_{E^*\times E}\leq \varphi(v)-\varphi(u)\text{ for all }v\in E \}
\]is called the subdifferential of $\varphi$. Any element $u^*\in \partial_c\varphi(u)$ is called a subgradient of $\varphi$ at $u$. It is a well know fact that $\partial\varphi_c$ is a maximal monotone operator.

Let $\Phi:E\rightarrow\mathbb R$ be a locally Lipschitz continuous functional and $u,\,v\in E$. We denote by $\Phi^\circ(u;v)$ the generalized Clarke directional derivative of $\Phi$ at the point $u$ in the direction $v$ defined by
\[
\Phi^\circ(u;v)=\displaystyle\limsup_{w\to u,\,t\downarrow 0}\frac{\Phi(w+tv)-\Phi(w)}{t}
\]
The generalized Clarke gradient $\partial \Phi:E\rightarrow 2^{E^*}$ of $\Phi$ at $u\in E$ is defined by
\[
\partial \Phi(u)=\{\xi\in E^*:\,\la\xi,v\ra_{E^*\times E}\leq \Phi^\circ(u;v)\text{ for all } v\in E\}
\]
We collect the following properties
\begin{enumerate}
\item[(a)] the function $v\mapsto \Phi^\circ(u;v)$ is positively homogeneous, subadditive and satisfies
\[
|\Phi^\circ(u;v)|\leq L_u\|v\|_E\text {for all } v\in E
\]where $L_u>0$ is the rank of $J$ near $u$.
\item[(b)] $(u,v)\mapsto \Phi^\circ(u;v)$ is upper semicontinuous.
\item[(c)] $\partial \Phi(u)$ is a nonempty, convex and weakly$^*$ compact subset of $E^*$ with $\|\xi\|_{E^*}\leq L_u$ for all $\xi\in\partial \Phi(u)$.
\item[(d)] for all $v\in E$, we have $\Phi^\circ(u;v)=\max\{\la\xi,v\ra_{E^*\times E}:\,\xi\in\partial \Phi(u)\}$.
\item[(e)] Let $F$ be another Banach spaces and $\\mathfrak t\in\mathcal L(F,E)$. Then
\begin{enumerate}
\item[(i)] $(\Phi\circ \mathfrak t)^\circ(u;v)\leq \Phi^\circ(\mathfrak tu;\mathfrak tv)$ for $u,\,v\in E$.
\item[(ii)] $\partial(\Phi\circ \mathfrak t)(u)\subseteq \mathfrak t^*\partial \Phi(\mathfrak tu)$ for $u\in E$ and where $\mathfrak t^*\in\mathcal L(E^*,F^*)$ denotes the adjoint operator to $\mathfrak t$.
\end{enumerate}

\end{enumerate}

The following surjectivity result for operators which are $L-$pseudomonotone will be used in our existence theorem in section 4 (cf. \cite[Theorem 2.1]{PPR99}).
\begin{theorem}\label{thm0}
If $E$ is a reflexive strictly convex Banach space, $L:D(L)\subset E\rightarrow E^*$ is a linear maximal monotone operator, and $A:E\rightarrow 2^{E^*}$ is a multivalued operator , which is bounded, coercive and $L-$pseudomonotone. Then $L+A$ is a surjective operator, i.e. for all $f\in E^*$, there exists $u\in E$ such that $Lu+Au\ni f$. 
\end{theorem}

It is worth to mention that one can drop the strict convexity of the reflexive Banach space $E$. It suffices to invoke the Troyanski renorming theorem to get an equivalent norm so that the space itself and its dual are strictly convex(cf. \cite[Proposition 32.23, p.862]{Z90}).

\section{An Existence result}

Let $\Omega\subset\mathbb R^N$ be a bounded domain with Lipschitz boundary $\Gamma:=\partial\Omega$. Let $\lambda_N$ denote the $N-$dimensional Lebesgue measure and $\sigma$ the surface measure on $\Gamma$. For simplicity, we take $b=1$ and $a$ the $N-$dimensional matrix identity. Define the following product space
\[
\mathbb H=\{U=(u_1,u_2):\, u_1\in L^2(\Omega),\,u_2\in L^2(\Gamma)\}
\]endowed with the inner product
\[
\la U,V\ra_{\mathbb H}=\la u_1,v_1\ra_{L^2(\Omega)}+\la u_2,v_2\ra_{L^2(\Gamma)},\forall \,U=(u_1,u_2),\,V=(v_1,v_2)\in\mathbb H
\]and the induced natural norm $|.|:=\la.,.\ra_{\mathbb H}^{1/2}$. Set $\mu=\lambda_N\oplus \sigma$. Then $\mathbb H$ can be identified with $L^2(\ol\Omega,\mu)$. Identifying each function $u\in W^{1,2}(\Omega)$ with $U=(u_{|\Omega},u_{|\Gamma })$, one has that, $W^{1,2}(\Omega)$ is a dense subspace of $\mathbb H$. Define the Banach space
\[
\mathbb V=\{U=(u,u_{|\Gamma}):\,u\in W^{1,2}(\Omega)\}
\]and endow it with the norm
\[
\|U\|=\|u\|_{W^{1,2}(\Omega)}+\|u\|_{L^2(\Gamma)}
\] for all $U=(u,u_{|\Gamma})$ and $V=(v,v_{|\Gamma})$ is $\mathbb V$. It is easy to see that we can identify $\mathbb V$ with $W^{1,2}(\Omega)\oplus L^2(\Gamma)$ under this norm. Moreover, we emphasize that $\mathbb V$ is not a product space and since $W^{1,2}(\Omega)\hookrightarrow L^2(\Gamma)$ by trace theory $\mathbb V$ is topologically isomorphic to $W^{1,2}(\Omega)$ in the obvious way. It is also immediate that  $\mathbb V$ is compactly embedded into $\mathbb H$. We have then the Gelfand triple
\[
\mathbb V\subset\mathbb H\subset\mathbb V^*
\]with continuous and compact embeddings. The embedding $\Lambda:\mathbb V\rightarrow\mathbb H$ is defined in a natural way by $\Lambda(U)=(i(u),\gamma(u))$, where $i:W^{1,2}(\Omega)\rightarrow L^2(\Omega)$ is the natural embedding and $\gamma$ is the trace operator. It is obvious that $\Lambda$ is continuous and compact from $\mathbb V$ into $\mathbb H$. Consider the Laplacian operator  with multivalued nonmonotone dynamic boundary conditions described as follows
\begin{equation}\label{eq1}
\begin{cases}
\partial_t u-\Delta u+\partial \phi_1(u)\ni f_1,&\quad \text{in } [0,T]\times\Omega\\
\partial_t u+\partial_\nu u+au+\partial \phi_2(u)\ni f_2&\quad \text{in }[0,T]\times\Gamma\\
u(0)=u_0 &\quad\text{in }\Omega
\end{cases}
\end{equation}where $a\in L^\infty(\partial\Omega)$ with $a\geq a_0>0$ for some constant $a_0$ and $\phi_1,\,\phi_2$ are locally Lipschitz functions on $\mathbb R$. Let $\la .,.\ra$ denotes the duality between $\mathbb V$ and $\mathbb V^*$. The system \eqref{eq1} can be written as follows
\begin{equation}\label{eq2}
\begin{cases}
\partial_t U+A U+\partial \phi_1(u)\times\partial \phi_2(u)\ni (f_1,f_2)\\
U(0)=U_0
\end{cases}
\end{equation} where the operator $A:\mathbb V\rightarrow \mathbb V^*$ is defined by $AU=(-\Delta  u,\partial_\nu u+au)$ so that
\[
\la A(U),V\ra=\int_\Omega \nabla u.\nabla v\,dx+\int_\Gamma auv\,d\sigma\]for $U=(u,u_{|\Gamma}),\,V=(v,v_{|\Gamma})\in\mathbb V$. From \cite[Lemma 2.111]{CLM07}, it is clear that the operator $A$ is pseudomonotone. The continuity and coercivity of $A$ can be proved in the same way as for Robin boundary conditions. Now, by using the definition of Clarke subdifferential, the system \eqref{eq2} leads to
\begin{equation}
\begin{cases}
\la U'+A U-f,V-U\ra+\int_{\Omega}\phi_1^\circ(u;v-u)\,dx+\int_{\partial\Omega}\phi_2^\circ(u;v-u)\,d\sigma\geq 0\\
U(0)=U_0
\end{cases}
\end{equation}for all $U=(u,u_{|\Gamma}),\,V=(v,v_{|\Gamma})\in\mathbb V$ where $f=(f_1,f_2)\in\mathbb V^*$. Define $\phi:\mathbb V\rightarrow \mathbb R$ by $\phi(U)=\phi_1(u)+\phi_2(u_{|\Gamma})$ for $U=(u,u_{|\partial\Omega})$. It follows that $\partial \phi(U)\subset \partial \phi_1(u)\times\partial \phi_2(u_{|\partial\Omega})$. Then the problem \eqref{eq1} has a solution if the following problem has one
\begin{equation}\label{eqJ}
\begin{cases}
U'+A U+\Lambda^*\partial \phi(\Lambda(U))\ni f\\
U(0)=U_0
\end{cases}
\end{equation}where $\Phi(U)=\int_{\ol\Omega}\phi(U)\,d\mu$. It is clear that the equivalence holds, if $\phi_1$ is regular at $u$ or $\phi_2$ at $u_{|\Gamma}$. An equivalent formulation to \eqref{eqJ} reads: for every $V\in\mathbb V$,
\begin{equation}
\begin{cases}
\la U'+ A U-f,V-U\ra+\int_{\ol\Omega}\phi^\circ(U;V-U)\,d\mu\geq 0\\
U(0)=U_0
\end{cases}
\end{equation}

In what follows we need the spaces $\mathscr V=L^2(0,T;\mathbb V)$, $\mathscr H=L^2(0,T;\mathbb H)$ and $\mathscr W=\{w\in\mathscr V:\,w'\in\mathscr V^*\}$ where the time derivative involved in the definition of $\mathscr W$ is understood in the sense of vector valued distributions. As usual, we equip it with the norm $\|w\|_{\mathscr W}=\|w\|_{\mathscr V}+\|w'\|_{\mathscr V^*}$, which makes the space $\mathscr W$  a separable Banach space. Clearly $\mathscr W\subset\mathscr V\subset\mathscr H\subset\mathscr V^*$. Moreover, we denote the duality for the pair $(\mathscr V,\mathscr V^*)$ as follows
\[
\la\la f, V\ra\ra=\int_0^T\la f(t),V(t)\ra\,dt
\]for $f\in\mathscr V^*$ and $V\in\mathscr V$. It is know \cite{Z90} that the embedding $\mathscr W\subset C(0,T;\mathbb H)$ is continuous. The problem under consideration is as follows: find $U\in\mathscr W$ such that for all $V\in\mathbb V$ and a.e. $t\in(0,T)$
\begin{equation}\label{problem00}
\begin{cases}
\la U'(t)+A U(t)-f(t),V-U(t)\ra+\int_{\ol\Omega}\phi^\circ(t,x,U(t);V-U(t))\,d\mu(x)\geq 0\\[0.2cm]
U(0)=U_0
\end{cases}
\end{equation}

We will prove the existence of solutions to the heat problem with multivalued nonmonotone dynamic boundary condition by considering  functionals defined in $L^2(\ol\Omega)$. Define the functional $\Phi:(0,T)\times L^2(\ol\Omega,d\mu)\rightarrow\mathbb R$ by
\begin{equation}\label{eqj}
\Phi(t,U)=\int_{\ol\Omega}\phi(t,U(x))\,d\mu(x),\quad t\in(0,T),\,U\in\mathbb V.
\end{equation}
Let us consider the following hypotheses 

\begin{enumerate}
\item[~]$H(\phi)$\quad $\phi:(0,T)\times\ol\Omega\times\mathbb R^2\rightarrow\mathbb R$ is a function such that\\
~
\begin{enumerate}
\item[(i)] $\phi(t,.,\xi)$ is measurable for all $t\in(0,T)$, $\xi=(\xi_1,\xi_1)\in\mathbb R^2$ and $\phi_1(t,.,0)\in L^1(\ol\Omega)$.
\item[(ii)] $\phi(t,x,.)$ is locally Lipschitz for all $\in(0,T)$, $x\in\ol\Omega$
\item[(iii)] $(\nu_1,\nu_1)\in\partial \phi(t,x,(\xi_1,\xi_2)\Rightarrow |(\nu_1,\nu_2)|_{\mathbb R^2}\leq c(1+|(\xi_1,\xi_2)|_{\mathbb R^2})$ for all $t\in(0,T)$, $x\in\ol\Omega$ with $c>0$.
\item[(iv)] $\phi^\circ(t,x,(\xi_1,\xi_2),-(\xi_1,\xi_2))\leq d(1+|(\xi_1,\xi_2)|_{\mathbb R^2})$ for all $t\in(0,T)$, $x\in\ol\Omega$ with $d\geq 0$.
\end{enumerate}
\end{enumerate}
\begin{enumerate}
\item[~] $(H_0)$\quad $U_0\in \mathbb H$, $f\in\mathscr V^*$.
\end{enumerate}
One can see that if $\phi_1$ and $\phi_2$ satisfy assumptions similar to $H(\phi)$, then $H(\phi)$ holds. The following lemma will be proved in the same way as the similar one for functionals on $\Gamma$(cf. \cite{MO04,PH08}).
\begin{lemma}\label{lem1}
Assume $\phi:(0,T)\times\ol\Omega\times\mathbb R^2\rightarrow\mathbb R$ satisfies hypothesis $H(\phi)$. Then the functional $\Phi$ give by \eqref{eqj} is well defined and locally Lipschitz in the second variable(in fact, Lipschitz in the second variable on bounded subsets of $\mathbb H$, its generalized gradient satisfies the linear growth condition 
\[
(\xi_1,\xi_2)\in\partial \Phi(t,V)\Rightarrow \|(\xi_1,\xi_2\|_{\mathbb H}\leq c'(1+\|V\|_{\mathbb H})
\]with $c'>0$ and for its generalized directional derivative we have 
\[
\Phi^\circ(t,U;V)\leq \int_{\ol\Omega} \phi^\circ(t,x,U(x);V(x))\,d\mu(x)
\]for $t\in(0,T)$, $U,\,V\in \mathbb H$ and 
\[
\Phi^\circ(t,U;-U)\leq d'(1+\|U\|_{\mathbb H})
\]with $d'>0$.
\end{lemma}
From \eqref{inclusion}, it is clear that in order to obtain the solvability of the problem \eqref{problem00}, it is enough to show that the problem 
\begin{equation}\label{prob1}
\begin{cases}
\partial_t U+ A U+\Lambda^*\partial \phi(\Lambda U)\ni f\\
U(0)=U_0
\end{cases}
\end{equation}admits a solution. The proofs are similar to the ones in \cite{MO04,PH08}.

\begin{proposition}
Suppose that hypotheses $H(\phi)$ and $H_0$ hold and $U$ is a solution of \eqref{prob1}, then there exists a constant $C>0$ such that 
\begin{equation}
\|U\|_{\mathscr W}\leq C(1+\|U_0\|_{\mathbb H}+\|f\|_{\mathscr V})
\end{equation}
\end{proposition}

\begin{theorem}
If hypotheses $H(\phi)$ and $H_0$ hold, then problem \eqref{prob1} has a solution.
\end{theorem}
\begin{proof}
By the density of $\mathbb V$ in $\mathbb H$, we may assume that $U_0\in\mathbb V$. Define the Nemytskii operators corresponding to $\mathcal A:\mathscr V\rightarrow\mathscr V^*$ and $\mathcal N:\mathscr V\rightarrow 2^{\mathscr V^*}$ as follow
\[
(\mathcal A V)(t)= AV(t)+AU_0
\]and
\[
\mathcal NV=\{\omega\in \mathscr H:\,\omega(t)\in\Lambda^*\partial \Phi(t,\Lambda(V(t)+U_0)),\,\text{ a.e. } t\in(0,T)\}
\]for all $V\in \mathscr V$. Problem \eqref{prob1} reads
\begin{equation}\label{prob2}
\begin{cases}
Z'(t)+\mathcal A Z(t)+\mathcal N Z(t)\ni f\\
Z(0)=0
\end{cases}
\end{equation}
We note that $Z\in\mathscr W$ is a solution to problem \eqref{prob2} if and only if $Z+U_0\in\mathscr W$ is a solution to problem \eqref{prob1}. Let $\mathcal L:D(\mathcal L)\subset\mathscr V\rightarrow \mathscr V^*$ be the operator defined by $\mathcal LV=V'$ with $D(\mathcal L)=\{w\in\mathscr W: w(0)=0\}$. It is well known (cf. \cite{Z90}) that $\mathcal L$ is a linear densely defined and maximal monotone operator. As a consequence, from\eqref{prob2} we obtain the problem
\begin{equation}\label{prob3}
\text{find }Z\in D(\mathcal L):\qquad \mathcal LZ+\mathcal TZ\ni f
\end{equation}where $\mathcal T:\mathscr V\rightarrow 2^{\mathscr V^*}$ is the operator given by $\mathcal T=\mathcal A+\mathcal N$. It is clear that problems \eqref{prob3} and \eqref{prob1} are equivalent. Now, to prove the existence of solutions to problem \eqref{prob3}, it suffices to use theorem \ref{thm0} and standard techniques from \cite{MO04,PH08}.

\end{proof}

\section{Partial generalized gradient}

Let $X$ be a Banach space. We say that a function $\phi:X\rightarrow \mathbb R$ is regular at $x$, if for all $v$, the usual one-sided directional derivative \[\phi'(x,v):=\displaystyle\lim_{h\downarrow0}\frac{\phi(x+hv)-\phi(x)}{h}\] exists and is equal to the generalized directional derivative $\phi^\circ(x;v)$.
Let $E=E_1\times E_2$, where $E_1$, $E_2$ are Banach spaces, and let $\phi:E\rightarrow\mathbb R$ be a locally Lipschitz function. We denote $\partial_1\phi(x_1,x_2)$ the partial generalized gradient of $\phi(.,x_2)$ at $x_1$, and by $\partial_2\phi(x_1,x_2)$ that of $\phi(x_1,.)$ at $x_2$. It is a fact that in general neither of the sets $\partial \phi(x_1,x_2)$ and $\partial_1 \phi(x_1,x_2)\times\partial_2 \phi(x_1,x_2)$ need be contained in the other. To be convinced it suffices to consider the function $f:\mathbb R^2\rightarrow\mathbb R$ defined by
\[
f(x,y)=(x\wedge (-y))\vee(y-x)
\]
From \cite[Example 2.5.2]{C83}, we have
\[
\partial_xf(0,0)\times\partial_yf(0,0)\not\subset\partial f(0,0)\not\subset \partial_xf(0,0)\times\partial_yf(0,0)
\]
For regular functions, however, a general relationship does hold between these sets. From \cite[Proposition 2.3.15]{C83} if $\phi$ is regular at $x=(x_1,x_1)\in E$, then
\begin{equation}\label{partial}
\partial \phi(x_1,x_2)\subset\partial_1 \phi(x_1,x_2)\times\partial \phi_2(x_1,x_2)
\end{equation}
and there is no reason that the equality holds even for regular functions. Next we will give a situation where the inclusion \eqref{partial} holds in a nonregular case, that is for functions with separated variables.

Consider two locally Lipschitz functions $\phi_1:E_1\rightarrow \mathbb R$ and $\phi_2:E_2\rightarrow \mathbb R$ and define the function $\phi:E=E_1\times E_2\rightarrow\mathbb R$ by 
\[
\phi(x_1,x_2)=\phi_1(x_1)+\phi_2(x_2)
\]We have the following result
\begin{proposition}\label{partial2}
The function $\phi$ is locally Lipschitz and for every $(x_1,x_2)\in E$, we have
\[
\partial \phi(x_1,x_2)\subset \partial \phi_1(x_1)\times\partial \phi_2(x_2)
\]Moreover, we have $\partial_k \phi=\partial \phi_k$ with $k=1,\,2$. If $\phi_1$ is regular at $x_1$ or $\phi_2$ at $x_2$, then equality holds.
\end{proposition}
\begin{proof}
Let $x=(x_1,x_2),\,y=(y_1,y_2)\in V_1\times V_2$, where $V_1\times V_2$ is some neighbourhood in $E$. Then
\begin{align*}
|\phi(y)-\phi(x)|\leq &|\phi_1(y_1)-\phi_1(x_1)|+|\phi_2(y_2)-\phi_2(x_2)|\\
\leq & K_1\|y_1-x_1\|_{E_1}+K_2\|y_2-x_2\|_{E_2}\\
\leq & (K_1+K_2)\sqrt{\|y_1-x_1\|_{E_1}^2+\|y_2-x_2\|_{E_2}^2}
\end{align*}
Let $z=(z_1,z_2)\in\partial j(x_1,x_2)$. Then for every $v=(v_1,v_2)\in E$ we have by definition
\[
\la z,v\ra_{E\times E^*}=\la z_1,v_1\ra_{E_1\times E_1^*}+\la z_2,v_2\ra_{E_2\times E_2^*}\leq \phi^\circ(x;v),
\]On the other hand, we get
\begin{align*}
\phi^\circ(x;v)=&\displaystyle\limsup_{(y_1,y_2)\to (x_1,x_2),h\downarrow 0}\frac{\phi((y_1,y_2)+h(v_1,v_2))-\phi(y_1,y_2)}{h}\\
=& \displaystyle\limsup_{(y_1,y_2)\to (x_1,x_2),h\downarrow 0}\frac{\phi_1(y_1+hv_1)-\phi_2(y_1)+\phi_2(y_2+hv_2)-\phi_2(y_2)}{h}\\
\leq& \displaystyle\limsup_{y_1\to x_1,h\downarrow 0}\frac{\phi_1(y_1+hv_1)-\phi_2(y_1)}{h}+\displaystyle\limsup_{y_2\to x_2,h\downarrow 0}\frac{\phi_2(y_2+hv_2)-\phi_2(y_2)}{h}\\
=& \phi_1^\circ(x_1;v_1)+\phi_2^\circ(x_2;v_2)
\end{align*}
It follows that for every $(v_1,v_2)\in E$,
\[
\la z_1,v_1\ra_{E_1\times E_1^*}+\la z_2,v_2\ra_{E_2\times E_2^*}\leq  \phi_1^\circ(x_1;v_1)+\phi_2^\circ(x_2;v_2)
\]
We take $v_2=0$, then for every $v_1\in E_1$ we have
\[
\la z_1,v_1\ra_{E_1\times E_1^*}\leq  \phi_1^\circ(x_1;v_1)
\]which means that $z_1\in\partial \phi_1(x_1)$. Analogously, we obtain that $z_2\in\partial \phi_2(x_2)$. Now, if $\phi_1$ is regular at $x_1$, we have
\begin{equation}\label{reg}
\phi^\circ(x;v)=\phi_1'(x_1;v_1)+\phi_2^\circ(x_2;v_2)=\phi_1^\circ(x_1;v_1)+\phi_2^\circ(x_2;v_2)
\end{equation}If $(z_1,z_2)\in \partial \phi_1(x_1)\times\partial \phi_2(x_2)$, then by \eqref{reg} we get $\la(z_1,z_2),(v_1,v_2)\ra_{E\times E^*}=\la z_1,v_1\ra_{E_1\times E_1^*}+\la z_2,v_2\ra_{E_2\times E_2^*}\leq \phi_1^\circ(x_1;v_1)+\phi_2^\circ(x_2;v_2)=\phi^\circ(x;v)$. Which means that $(z_1,z_2)\in \partial \phi(x_1,x_2)$. Similarly, if $\phi_2$ is regular at $x_2$, then equality holds.
\end{proof}
\begin{remark}If $\phi_1$ and $\phi_2$ are convex, then $j$ is convex and the generalized directional derivative coincides with the one sided directional derivative, then it is clear that \eqref{partial} in the general case holds and the equality in Proposition \ref{partial2} holds too.
\end{remark}

The well known Chang's Lemma (cf. \cite[Example 1]{chang80}) concerns the calculation of the Clarke's gradient of a locally Lipschitz function
$\phi:\mathbb R\rightarrow\mathbb R$ defined by
\[\phi(t)=\int_0^t\gamma(\xi)d\xi\]
where $\beta\in L^{\infty}_{loc}(\mathbb R)$. Consider the functions
\[\overline\gamma_\mu(t)=\esssup\limits_{|s-t|<\mu}\gamma(s)\quad \text{ and } \quad\underline\gamma_\mu(t)=\essinf\limits_{|s-t|<\mu}\gamma(s)\]
They are increasing and decreasing functions of $\mu$, respectively. Therefore, the limits for $\mu\to 0^+$ exist. We denote them by $\overline\gamma(t)$ and $\underline\gamma(t)$, respectively. Then it is proved by Chang that
\[\partial \phi(t)\subset[\underline\gamma(t),\overline\gamma(t)].\]
If in addition $\gamma(t\pm 0)$ exists for every $t\in\mathbb R$ then, the equality holds, i.e.
\[\partial \phi(t)=[\underline\gamma(t),\overline\gamma(t)].\]
Now, let $\gamma_1,\,\gamma_2\in L^{\infty}_{loc}(\mathbb R)$ and define a locally Lipschitz function $\phi:\mathbb R^2\rightarrow\mathbb R$ as follows
\[
\phi(t,s)=\int_0^t\gamma_1(\xi)d\,\xi+\int_0^s\gamma_2(\xi)d\,\xi
\]
\begin{theorem}[Chang's type Lemma]
One has
\begin{equation}\label{partial3}
\partial \phi(t,s)\subset\widehat{\gamma}_1(t)\times\widehat{\gamma}_2(s)
\end{equation}
If in addition $\gamma_1(t\pm 0)$ and $\gamma_2(s\pm 0)$ exist for every $t,s$ then
\[
\partial \phi(t,s)=\widehat{\gamma}_1(t)\times\widehat{\gamma}_2(s)
\]
\end{theorem}
\begin{proof}
By proposition \ref{partial2} and the classical Chang's lemma discussed above, the inclusion \eqref{partial3} holds true. If further, we suppose that, for $k=1,\,2$, $\gamma_k(t\pm 0)$ exists for each $t\in\mathbb R$, then $\ul\gamma_k(t)=\min\{\gamma_k(t+ 0),\gamma_k(t- 0)\}$ and $\ol\gamma_k(t)=\max\{\gamma_k(t+ 0),\gamma_k(t- 0)\}$. By the definition of Clarke directional derivative at $(t,s)$ in the direction $(z_1,z_2)$, we have
\begin{align*}
\phi^\circ(t,s;z_1,z_2)=&\limsup\limits_{h_1,\,h_2\to0,\,\lambda\downarrow 0} \frac{1}{\lambda} \Big( \phi(t+h_1+\lambda z_1,s+h_2+\lambda z_2) - \phi(t+h_1,s+h_2) \Big) \\
=& \limsup\limits_{h_1,\,h_2\to0,\,\lambda\downarrow 0}\frac{1}{\lambda}\left[\int_{t+h_1}^{t+h_1+\lambda z_1}\gamma_1(\tau)\,d\tau+\int_{s+h_2}^{s+h_2+\lambda z_2}\gamma_2(\tau)\,d\tau\right]\\
\geq & \lim\limits_{h_1\to 0,\,\lambda\downarrow 0}\frac{1}{\lambda}\int_{t+h_1}^{t+h_1\lambda z_1}\gamma_1(\tau)\,d\tau+\lim\limits_{h_2\to 0,\,\lambda\downarrow 0}\frac{1}{\lambda}\int_{s+h_2}^{s+h_2+\lambda z_2}\gamma_2(\tau)\,d\tau\\
\geq &  \lim\limits_{h_1\to 0,\,\lambda\downarrow 0} z_1\int_{0}^{1}\gamma_1(t+h_1+\lambda z_1\tau)\,d\tau+\lim\limits_{h_2\to 0,\,\lambda\downarrow 0}z_2\int_{0}^{1}\gamma_2(s+h_2+\lambda z_2\tau)\,d\tau\\
\geq & \gamma_1(t\pm 0)z_1+\gamma_2(s\pm 0)z_2
\end{align*}
It follows then that 
\[
\big(\gamma_1(t\pm 0),\gamma_1(s\pm 0)\big)\in\partial \phi(t,s),\quad \text{ for all }(t,s)\in\mathbb R^2
\]
Since $\partial \phi(t,s)$ is convex, then the equality in \eqref{partial3} holds true.
\end{proof}

\section{Galerkin approximation}

The aim of this section is to consider the convergence of a numerical approximation constructed by Galerkin method. This method can also be an alternative way to prove the existence result in section 3, without making use of surjectivity results and pseudomonotone operators theory.

Let $\gamma_1,\,\gamma_2\in L^\infty_{\mathrm{loc}}(\mathbb R)$ and for $k=1,\,2$ define $\phi:\mathbb R^2\rightarrow\mathbb R$ by $\phi(t,s)=\phi_1(t)+\phi_2(t)$ for all $(t,s)\in \mathbb R^2$ where \[\phi_k(t)=\int_0^t\gamma_k(s)\,ds,\quad\text{ for all }t\in\mathbb R\]
Let $p\in C_0^\infty(\mathbb R)$ be a positive function with support in $[-1,1]$ such that $\int_{\mathbb R}p(\xi)\,d\xi=1$. For $\xi\in\mathbb R$ and $\varepsilon>0$, define the function $p_\varepsilon(\xi)=\frac{1}{\varepsilon}p(\frac{\xi}{\varepsilon})$ and for $k=1,\,2$ define

\[
\gamma_{k\varepsilon}(\xi)=\int_{\mathbb R}p_\varepsilon(\eta)\gamma_k(\xi-\eta)\,d\eta,\quad \xi\in\mathbb R, \varepsilon>0 
\]
We consider a Galerkin basis $\{Z_1,Z_2,\dots\}$ of $\mathbb V$ and let $\mathbb V_m=\mathrm{span}\{Z_1,Z_2,\dots,Z_m\}$ be the resulting $m-$dimensional subspaces. Let $\{U_{m0}\}_m$ be an approximation of the initial value $U_0$ such that $U_{m0}\in\mathbb V_m$, $U_{m0}\rightarrow U_0$ in $\mathbb H$ and $\{U_{m0}\}_m$ is bounded in $\mathbb V$. Let $\{\varepsilon_m\}_m$ be a sequence of real numbers converging to zero as $m\to\infty$. Instead of $\gamma_{k\varepsilon_m}$ we will write $\gamma_{km}$ and we will use the notation \[\gamma_m(t,s)=(\gamma_{1m}(t),\gamma_{2m}(s))\quad \forall(t,s)\in\mathbb R^2.\]

We consider the following regularized Galerkin system of finite dimensional differential equations: find $U_m=(u_m,u_{m|\Gamma})\in L^2(0,T;\mathbb V_m)$ with $U_m'\in L^2(0,T;\mathbb V_m)$ such that 
\begin{equation}\label{approx1}
\begin{cases}
\la U_m'(t)+AU_m(t),V\ra+\la\gamma_{m}(U_m),V\ra_{\mathbb H}=\la f(t),V\ra\\
U_m(0)=U_{m0}
\end{cases}
\end{equation}for a.e. $t\in(0,T)$ and for all $V=(v,v_{|\Gamma})\in\mathbb V$. The problem \eqref{approx1} can be written more explicitly as follows
\begin{equation}\label{approx2}
\begin{cases}
\la U_m'(t)+AU_m(t) ,V\ra+\int_{\Omega}\gamma_{1m}(u_m(t)).vdx+\int_{\Gamma}\gamma_{2m}(u_m(t)).vd\sigma=\la f(t),V\ra\\
U_m(0)=U_{m0}
\end{cases}
\end{equation}for a.e. $t\in(0,T)$ and for all $V\in\mathbb V$.

For the existence of solutions we will need the following hypothesis $H(\gamma)$ : for $k=1,\,2$ assume that 

\begin{itemize}
\item[~ ](Chang condition) $\gamma_k \in L_{loc}^{\infty}(\real),\, \gamma_k(t \pm 0)\text{ exists for any } t \in \real.$
\item[~ ] (Growth condition) for all $t\in \mathbb R$ we have
\[
|\gamma_k(t)|\leq c_k(1+|t|^{\theta_k})
\]with $c_k>0$ and $0\leq\theta_k\leq 1$.
\end{itemize}

\begin{theorem}
Let $H(\gamma)$ holds. Moreover, assume that one of the following situations holds
\begin{enumerate}\item[~]
\begin{enumerate}
\item[$\mathrm{(1)}$] $\theta_1,\,\theta_2<1$
\item[$\mathrm{(2)}$] $\theta_1<1$ and $\theta_2=1$ provided $c_2<\frac{M}{2\sqrt 2}$
\item[$\mathrm{(3)}$] $\theta_2<1$ and $\theta_1=1$ provided $c_1<\frac{M}{2\sqrt 2}$
\item[$\mathrm{(4)}$] $\theta_1=\theta_2=1$ provided $c_1+c_2<\frac{M}{2\sqrt 2}$
\end{enumerate}
\end{enumerate}
where $M$ is the coercivity constant of the operator $A$. Then problem \eqref{prob1} has at least one solution.
\end{theorem}
\begin{proof}
We substitute $U_m(t)= \sum_{k=1}^m c_{km}(t)Z_k$ in \eqref{approx2} to obtain an initial value problem for a system of first order ordinary differential equations for $c_{km}$, $k=1,\dots,m$, where the initial values $c_{km}(0)$ are given by $U_{m0}=\sum_{k=1}^m c_{km}(0) Z_k$. From the Caratheodory theorem, the solution $U_m$ exists on $[0,t_{\mathrm{max}})$, and we can extended it on the closed interval $[0,T]$ by using a priori estimates below. By replacing $V$ with $U_m$ in \eqref{approx2} we get for a.e. $t\in(0,T)$
\begin{align*}
\la U_m'(t)+AU_m(t) ,U_m(t)\ra+&\int_{\Omega}\gamma_{1m}(u_m).u_mdx\\
&+\int_{\Gamma}\gamma_{2m}(u_m).u_md\sigma=\la f(t),U_m(t)\ra
\end{align*}
Using the coercivity of $A$ and the Young inequality, we have
\begin{equation}
\frac{1}{2}\frac{d}{dt}|U_m(t)|^2+M\|U_m(t)\|^2+\la\gamma_{m}(U_m(t)),U_m(t)\ra_{\mathbb H}\leq\frac{1}{2M}\|f(t)\|_{\mathbb V^*}^2+\frac{M}{2}\|U_m(t)\|^2
\end{equation}for a.e. $t\in(0,T)$. Consequently
\begin{equation}
\frac{1}{2}\frac{d}{dt}|U_m(t)|^2+\frac{M}{2}\|U_m(t)\|^2+\la\gamma_{m}(U_m(t)),U_m(t)\ra_{\mathbb H}\leq\frac{1}{2M}\|f(t)\|_{\mathbb V^*}^2
\end{equation}
Integrating over $(0,t)$, we get
\begin{equation}
\frac{1}{2}|U_m(t)|^2-\frac{1}{2}|U_{0m}|^2+\frac{M}{2}\int_0^t\|U_m(s)\|^2\,ds+\int_0^t\la\gamma_{m}(U_m(s)),U_m(s)\ra_{\mathbb H}\,ds\leq\frac{1}{2M}\int_0^t\|f(s)\|_{\mathbb V^*}^2\,ds
\end{equation}
From the growth condition, we have 
\[
|\gamma_{k\varepsilon}(s)|\leq c_k(1+|s|^{\theta_k})
\] for $k=1,\,2$ and $s\in\mathbb R$. It follows that 
\begin{align*}
\int_{\Omega}|\gamma_{1m}(u_m(s,x))|^2\,dx\leq& c_1\int_\Omega(1+|u_m(s,x)|^{\theta_1})^2\,dx\\
\leq & 2c_1^2\int_{\Omega}(1+|u_m(s,x)|^{2\theta_1})\,dx\\
\leq & 2c_1^2\lambda_N(\Omega)+2c_1^2\lambda_N(\Omega)^{1-\theta_1}\|u_m(t)\|_{L^2(\Omega)}^{2\theta_1}\\
\leq & 2c_1^2\lambda_N(\Omega)+2c_1^2\lambda_N(\Omega)^{1-\theta_1}\|U_m(t)\|^{2\theta_1}
\end{align*}
Consequently
\begin{align*}
\|\gamma_{1m}(u_m)\|^2_{L^2(0,t;L^2(\Omega))}=&\int_0^t \|\gamma_{1m}(u_m(\tau))\|^2_{L^2(\Omega)}\,d\tau\\
\leq & 2tc_1^2\lambda_N(\Omega)+2c_1^2\lambda_N(\Omega)^{1-\theta_1}\int_0^t\|U_m(\tau)\|^{2\theta_1}\,d\tau\\
\leq & 2tc_1^2\lambda_N(\Omega)+2c_1^2\lambda_N(\Omega)^{1-\theta_1}t^{1-\theta_1}\|U_m\|_{L^2(0,t;\mathbb V)}^{2\theta_1}
\end{align*}
It then follows that
\[
\|\gamma_{1m}(u_m)\|_{L^2(0,t;L^2(\Omega))}\leq a_1+a_1' \|U_m\|_{L^2(0,t;\mathbb V)}^{\theta_1}
\]with $a_1=c_1\sqrt{2t\lambda_N(\Omega)}$ and $a'_1=c_1\sqrt{2\lambda_N(\Omega)^{1-\theta_1}t^{1-\theta_1}}$. This leads to
\begin{align*}
|\int_0^t\la\gamma_{1m}(u_m(\tau),u_m(\tau)\ra_{L^2(\Omega)\times L^2(\Omega)}\,d\tau|\leq &\int_0^t \|\gamma_{1m}(u_m(\tau))\|_{L^2(\Omega)}\|u_m(\tau)\|_{L^2(\Omega)}\,d\tau\\
\leq & \|\gamma_{1m}(u_m)\|_{L^2(0,t;L^2(\Omega))}\|u_m\|_{L^2(0,t;L^2(\Omega))}\\
\leq & (a_1+a_1' \|U_m\|_{L^2(0,t;\mathbb V)}^{\theta_1})\|U_m\|_{L^2(0,t;\mathbb V)}\\
\leq& a_1 \|U_m\|_{L^2(0,t;\mathbb V)}+a_1' \|U_m\|_{L^2(0,t;\mathbb V)}^{1+\theta_1}
\end{align*}
With same calculus we obtain
\begin{equation}
|\int_0^t\la\gamma_{2m}(u_m(\tau),u_m(\tau)\ra_{L^2(\Gamma)\times L^2(\Gamma)}\,d\tau|
\leq a_2 \|U_m\|_{L^2(0,t;\mathbb V)}+a_2' \|U_m\|_{L^2(0,t;\mathbb V)}^{1+\theta_2}
\end{equation}with $a_2=c_2\sqrt{2t\sigma(\Gamma)}$ and $a'_2=c_2\sqrt{2\sigma(\Omega)^{1-\theta_2}t^{1-\theta_2}}$. It follows that 
\begin{align*}
&\frac{1}{2}|U_m(t)|^2+\frac{M}{2}\int_0^t\|U_m(s)\|^2\,ds\leq\frac{1}{2M}\|f\|_{L^2(0,t;\mathbb V^*)}^2 +\frac{1}{2}|U_{0m}|^2\\
&+ a_1 \|U_m\|_{L^2(0,t;\mathbb V)}+a_1' \|U_m\|_{L^2(0,t;\mathbb V)}^{1+\theta_1}+a_2 \|U_m\|_{L^2(0,t;\mathbb V)}+a_2' \|U_m\|_{L^2(0,t;\mathbb V)}^{1+\theta_2}
\end{align*}
Which leads to
\begin{align*}
&\frac{1}{2}|U_m(t)|^2+\frac{M}{2}\|U_m\|_{L^2(0,t;\mathbb V)}^2\leq\frac{1}{2M}\|f\|_{L^2(0,t;\mathbb V^*)}^2 +\frac{1}{2}|U_{0m}|^2\\
&+ (a_1+a_2) \|U_m\|_{L^2(0,t;\mathbb V)}+a_1' \|U_m\|_{L^2(0,t;\mathbb V)}^{1+\theta_1}+a_2' \|U_m\|_{L^2(0,t;\mathbb V)}^{1+\theta_2}
\end{align*}
If $\theta_1,\,\theta_2<1$, it is clear that $\{U_m\}_m$ is bounded in $L^2(0,T;\mathbb V)$ with no additional conditions. Now if $\theta_1<1$ and $\theta_2=1$, then we have
\begin{align*}
&\frac{1}{2}|U_m(t)|^2+(\frac{M}{2}-a_2')\|U_m\|_{L^2(0,t;\mathbb V)}^2\leq\frac{1}{2M}\|f\|_{L^2(0,t;\mathbb V^*)}^2 +\frac{1}{2}|U_{0m}|^2\\
&\qquad+ (a_1+a_2) \|U_m\|_{L^2(0,t;\mathbb V)}+a_1' \|U_m\|_{L^2(0,t;\mathbb V)}^{1+\theta_1}
\end{align*}
As $1+\theta_1<2$, it follows that $\{U_m\}_m$ is bounded in $L^2(0,T;\mathbb V)$ provided $\frac{M}{2}-a_2'>0$. Similarly, if $\theta_1=1$ and $\theta_2<1$, it then follows that $\{U_m\}_m$ is bounded in $L^2(0,T;\mathbb V)$ provided $\frac{M}{2}-a_1'>0$. If $\theta_1=\theta_2=1$, then we get 
\begin{align*}
\frac{1}{2}|U_m(t)|^2+&(\frac{M}{2}-a_1'-a_2')\|U_m\|_{L^2(0,t;\mathbb V)}^2\leq\frac{1}{2M}\|f\|_{L^2(0,t;\mathbb V^*)}^2 +\frac{1}{2}|U_{0m}|^2\\
&\qquad+ (a_1+a_2) \|U_m\|_{L^2(0,t;\mathbb V)}
\end{align*}
It then follows that $\{U_m\}_m$ is bounded in $L^2(0,T;\mathbb V)$ provided $\frac{M}{2}-a_1'-a_2'>0$. As a summary, we conclude that $\{U_m\}_m$ is bounded in $L^2(0,T;\mathbb V)$ provided one of the following situations holds
\begin{enumerate}\item[~]
\begin{enumerate}
\item[(1)] $\theta_1,\,\theta_2<1$
\item[(2)] $\theta_1<1$ and $\theta_2=1$ provided $c_2<\frac{M}{2\sqrt 2}$
\item[(3)] $\theta_2<1$ and $\theta_1=1$ provided $c_1<\frac{M}{2\sqrt 2}$
\item[(4)] $\theta_1=\theta_2=1$ provided $c_1+c_2<\frac{M}{2\sqrt 2}$
\end{enumerate}
\end{enumerate}
When  $\{U_m\}_m$ is bounded in $L^2(0,T;\mathbb V)$ then it is also bounded in $L^\infty(0,T;\mathbb H)$, so passing to a subsequence, if necessary, we have 
\[
U_m\rightarrow U\text{ weakly in } L^2(0,T;\mathbb V) \text{ and weakly}-^* \text{ in } L^\infty(0,T;\mathbb H)
\]where $U\in L^2(0,T;\mathbb V)\cap L^\infty(0,T;\mathbb H)$. From the above estimates we have also that $\{U'_m\}_m$ is bounded in $L^2(0,T;\mathbb V^*)$. Thus by passing to a subsequence, if necessary, we get 
\[
U_m\rightarrow U \text{ weakly in }\mathscr W \text{ with }U\in\mathscr W
\]
On the other hand, as $U_m$ converges weakly to $U$ in $L^2(0,T;\mathbb V)$ and $\mathbb V\subset \mathbb H$ compactly then $U_m$ converges to $U$ in $L^2(0,T;\mathbb H)$, which means that
\[
u_m\rightarrow u \text{ in }L^2(0,T;L^2(\Omega))\text{ and } u_{m|_\Gamma}\rightarrow u_{|_\Gamma} \text{ in }L^2(0,T;L^2(\Gamma))
\]
Since the mapping $\mathscr W\ni W\mapsto W(0)\in\mathbb H$ is linear and continuous, we have $U_m(0)\rightarrow U(0)$ weakly in $\mathbb H$, which together with $U_{m0}\rightarrow U_0$ entails $U(0)=U_0$.
Let now $V\in\mathbb V$ and denote $\Psi_m(t,x)=\psi(t)V_m(x)$ where $\psi\in C_0^\infty(0,T)$ and $V_m\in\mathbb V_m$ is such that $V_m\rightarrow V$ in $\mathbb V$, we have $\Psi_m\rightarrow \Psi$ in $\mathscr W$ with $\Psi(t,x)=\psi(t)V(x)$. It follows that 
\[
\int_0^T\la U'(t)+AU(t),\Psi_m(t)\ra\,dt+\int_0^T\la\gamma_m(U_m(t)),\Psi_m(t)\ra_{\mathbb H}\,dt=\int_0^T\la f(t),\Psi_m(t)\ra\,dt
\]
Passing to the limit and remarking that $\psi$ is chosen arbitrary we deduce that 
\[
\la U'(t)+AU(t),V\ra+\la \xi(t),V\ra_{\mathbb H}=\la f(t),V\ra
\]where $\xi=(\xi_1,\xi_2)$. It remains to prove that $\xi_1\in\widehat\gamma_1(u(t,x))$ for a.e. $(t,x)\in(0,T)\times\Omega$ and $\xi_2\in\widehat\gamma_2( u_{|_\Gamma}(t,x))$ for a.e. $(t,x)\in(0,T)\times\Gamma$. We apply the convergence theorem of Aubin and Cellina \cite{AC84} to the multifunctions $\partial \phi_1$ and $\partial \phi_2$. First, we observe that $\partial \phi_1,\,\partial \phi_2:\mathbb R\rightarrow 2^{\mathbb R}$ are upper semicontinuous. Since $u_m\rightarrow u$ in $L^2(0,T;L^2(\Omega))$ and $ u_{m|_\Gamma}\rightarrow u_{|_\Gamma}$ in $L^2(0,T;L^2(\Gamma))$, then by the definition of $\widehat{\gamma}_1$ and $\widehat{\gamma}_2$, we deduce that for a.e $(t,x)\in(0,T)\times\Omega$ and for every neighbourhood $\mathcal N$ of zero in $\mathbb R^2$ there exists $n_0=n_0(t,x,\mathcal N)\in\mathbb N$ such that
\[
(u_m(t,x),\gamma_{1m}(u_m(t,x)))\in Gr\,\partial \phi_1+\mathcal N,\quad\text{for all }n\geq n_0
\]
By passing to the limit we get
\[
\xi_1(t,x)\in\ol{\mathrm{conv}}\partial \phi_1(u(t,x))=\partial \phi_1(u(t,x))
\]for a.e. $(t,x)\in(0,T)\times\Omega$. Analogously, we get
\[
\xi_2(t,x)\in\partial \phi_2( u_{|_\Gamma}(t,x))
\]for a.e. $(t,x)\in(0,T)\times\Gamma$, which completes the proof.

\end{proof}

\section{Concluding remarks}

In this paper, we introduced a new class of hemivariational inequalities, namely dynamic boundary hemivariational inequalities. It concerns dynamic boundary conditions with a Clarke subdifferential perturbation on the boundary. The suitable framework to study such problems is to work on a product space instead of the state space itself. We chose to work with dynamic boundary condition in its simplest form but we nevertheless could work with a general uniformly elliptic operator or even in the $L^p$ framework with the $p-$Laplacian. Moreover, with some changes on the choice of the product spaces, one can incorporate the Laplace-Beltrami operator on the boundary in addition of the usual Laplacian. On the other hand, one can replace the growth condition in  section 6 by the Rauch condition expressing the ultimate increase of the graphs of functions $\beta_k$.

As a continuation of this paper we aim to study the abstract version of the current work and to look at hemivariational inequalities that can be formulated in terms of matrix operators on product spaces. Indeed, this formulation covers a wide range of examples, including second order problems, equations with delay, equations which are memory dependent.

\begin{acknowledgements}
We would like to express our gratitude to the Editor for taking time to handle the manuscript and to anonymous referees whose constructive comments are very helpful for improving the quality of our paper. 

\end{acknowledgements}

\section*{Availability of data and materials}
Not applicable

\section*{Competing interests}
There is no competing of interests concerning this manuscripts.

\section*{Funding}
There is no funding regarding this work.

\section*{Authors' contributions}

Authors contributed equally in this work.

\end{document}